\documentclass[letterpaper,11pt,twoside,keywordsasfootnote,addressatend,noinfoline]{article}
\usepackage{fullpage}
\usepackage[english]{babel}
\usepackage{amssymb}
\usepackage{amsmath}
\usepackage{theorem}
\usepackage{epsfig}
\usepackage{subfigure}
\usepackage{mathrsfs}
\usepackage{color}
\usepackage{imsart}

\newtheorem{theorem}{Theorem}

\newtheorem{lemma}[theorem]{Lemma}

\newenvironment{proof}{\noindent{\bf Proof.}}{\hspace*{2mm}~$\square$}
\newenvironment{proofof}[1]{\noindent{\bf Proof of #1.}}{\hspace*{2mm}~$\square$}

\newcommand{\im}{\xi}
\newcommand{\imm}{X}
\newcommand{\at}{\zeta}
\newcommand{\att}{Z}

\newcommand{\N}{\mathbb{N}}
\newcommand{\Z}{\mathbb{Z}}
\newcommand{\R}{\mathbb{R}}

\newcommand{\A}{\mathscr{A}}
\newcommand{\C}{\mathscr{C}}
\newcommand{\E}{\mathscr{E}}
\newcommand{\F}{\mathscr{F}}
\newcommand{\G}{\mathscr{G}}
\newcommand{\V}{\mathscr{V}}

\renewcommand{\r}{\mathbf{r}}
\renewcommand{\d}{\mathbf{d}}
\newcommand{\ind}{\mathbf{1}}

\newcommand{\ep}{\epsilon}
\newcommand{\n}{\hspace*{-5pt}}

\DeclareMathOperator{\card}{card}

\DeclareMathOperator{\uniform}{Uniform \,}
\DeclareMathOperator{\sgn}{sgn}

\begin{document}

\begin{frontmatter}
\title     {Probability of consensus in spatial opinion \\ models with confidence threshold}
\runtitle  {Spatial opinion models with confidence threshold}
\author    {Mela Hardin and Nicolas Lanchier}
\runauthor {Mela Hardin and Nicolas Lanchier}
\address   {School of Mathematical and Statistical Sciences \\ Arizona State University \\ Tempe, AZ 85287, USA. \\ melahardin@asu.edu \\ nicolas.lanchier@asu.edu}

\maketitle

\begin{abstract} \ \
 This paper gives lower bounds for the probability of consensus for two spatially explicit stochastic opinion models.
 Both processes are characterized by two finite connected graphs, that we call respectively the spatial graph and the opinion graph.
 The former represents the social network describing how individuals interact, while the latter represents the topological structure of the opinion space.
 The representation of the opinions as a graph induces a distance between opinions which we use to measure disagreements.
 Individuals can only interact with their neighbors on the spatial graph, and each interaction results in a local change of opinion only if the
 two interacting individuals do not disagree too much, which is quantified using a confidence threshold.
 In the first model, called the imitation process, an update results in both neighbors having the exact same opinion, whereas in the second model,
 called the attraction process, an update results in the neighbors' opinions getting one unit closer.
 For both models, we derive a lower bound for the probability of consensus that holds for any finite connected spatial graph.
 For the imitation process, the lower bound for the probability of consensus also holds for any finite connected opinion graph, whereas for the attraction
 process, the lower bound only holds for a certain class of opinion graphs that includes finite integer lattices, regular trees and star-like graphs.
\end{abstract}



\end{frontmatter}


\section{Introduction}
\label{sec:intro}
 The simplest and most popular stochastic opinion model that includes space in the form of local interactions is the voter model introduced independently
 in~\cite{clifford_sudbury_1973, holley_liggett_1975}.
 In this model, individuals are represented by the vertex set of a connected graph and are characterized by one of two possible opinions.
 Pairs of neighbors (individuals connected by an edge) interact at rate one, which results in one of the two neighbors updating her opinion by imitating
 the other neighbor. \\
\indent The behavior of the voter model starting from a product measure on (infinite) integer lattices is now well understood.
 In one and two dimensions, the model clusters, i.e., the probability that any two individuals disagree tends to zero, whereas in
 higher dimensions, coexistence occurs in that the process converges weakly to a nontrivial stationary distribution~\cite{holley_liggett_1975}.
 In one dimension, the cluster size scales like the square root of time~\cite{bramson_griffeath_1980}, while in two dimensions, there is no natural scale
 for the cluster size~\cite{cox_griffeath_1986}.
 In addition, except in one dimension, the fraction of time any given individual holds a given opinion converges almost surely to the initial density of this
 opinion~\cite{cox_griffeath_1983}.
 In particular, in dimension two, referred to as the critical dimension for the voter model, clusters keep growing indefinitely while at the same time
 individuals switch opinions frequently.
 This apparent paradox indicates in fact that clusters move faster than they grow. \\
\indent The limiting behavior of the voter model on finite connected graphs is trivial.
 In this case, because the state space is finite, the process fixates in a configuration in which all the individuals share the same opinion:
 consensus occurs with probability one.
 In addition, a direct application of the optional stopping theorem for martingales shows that the probability that a given opinion wins is simply
 equal to the fraction of individuals who initially hold this opinion.
 In contrast, for other stochastic spatial opinion models such as the Deffuant model~\cite{deffuant_al_2001, lanchier_2012b}, the vectorial version of the
 Deffuant model~\cite{deffuant_al_2001, lanchier_scarlatos_2014}, and the constrained voter model~\cite{vazquez_krapivsky_redner_2003}, where the
 number of opinions is increased and a confidence threshold preventing interactions between individuals who disagree too much is imposed, whether
 a consensus is reached or the system fixates to a fragmented configuration with multiple discordant opinions is far from being obvious.
 The goal of this paper is to study this question for the general spatial opinion model introduced in~\cite{lanchier_scarlatos_2017} and
 a natural variant of this model.
 In particular, we give nontrivial lower bounds for the probability of consensus that are uniform in all possible finite connected graphs. \vspace*{-5pt} \\


\noindent {\bf Model description.}
 The first opinion model we study in this paper is the stochastic process that was introduced by Lanchier and Scarlatos
 in~\cite{lanchier_scarlatos_2017}, while the second model is a natural variant of the process~\cite{lanchier_scarlatos_2017}.
 Like the voter model, both processes keep track of the configuration of opinions in a population of individuals.
 Our models are spatial in that the individuals are represented by the vertices of a connected graph~$\G = (\V, \E)$, with the
 presence of an edge between two vertices indicating that the two individuals can interact.
 In particular, this connected graph has to be thought of as a social network.
 The analysis in~\cite{lanchier_scarlatos_2017} focuses on the case where this connected graph is the one-dimensional integer lattice and
 individuals can only interact with their two nearest neighbors.
 In contrast, we analyze the somewhat more realistic case where the graph can be any finite connected graph.
 Our models are also stochastic in the sense that individuals update their opinion at the times of independent Poisson processes by interacting
 with a random neighbor.
 In the voter model, individuals are characterized by one of two possible opinions and each interaction results in an agreement between neighbors.
 In contrast, we assume that the set of opinions is represented by the set of vertices of another finite connected graph~$G = (V, E)$.
 In particular, for both processes, the state at time~$t$ is a spatial configuration of opinions
 $$ \im_t : \V \to V \quad \hbox{where} \quad \im_t (x) = \hbox{opinion at vertex~$x$ at time~$t$}. $$
 The main reason for representing the set of opinions as the vertex set of a graph is to use the geodesic distance on this graph to measure
 the level of disagreement between any two individuals and incorporate a {\bf confidence threshold}~$\tau \in \N$ in the models:
 we assume that an interaction between two neighbors results in one of the two individuals updating her opinion if and only if the two individuals
 hold different opinions but the distance between their opinions before the interaction does not exceed the confidence threshold~$\tau$.
 This assumption is motivated by the psychological concept of homophily, the tendency for people to be attracted to those who are similar to themselves.
 To distinguish between the graph representing the individuals' social network and the graph representing the structure of the opinion space,
 we will call from now on
 $$ \begin{array}{rcl}
      \G = (\V, \E) & \n = \n & \hbox{{\bf spatial graph}} \vspace*{4pt} \\
         G = (V, E) & \n = \n & \hbox{{\bf opinion graph}}. \end{array} $$
 To describe the dynamics, it is convenient to think of the spatial graph as a directed graph with each edge in~$\E$ inducing two directed
 edges, i.e., we define
 $$ \vec{\E} = \{\vec{xy}, \vec{yx} : \{x, y \} = \{y, x \} \in \E \} = \hbox{set of directed edges}. $$
 In both opinion models, each directed edge, say~$\vec{xy}$, becomes active at rate one, which results in an update of the opinion at vertex~$x$
 if and only if vertices~$x$ and~$y$ disagree but the distance between their opinions does not exceed the threshold~$\tau$.
 In the first model~\cite{lanchier_scarlatos_2017}, denoted by~$(\xi_t)$ from now on, the update results in vertex~$x$ imitating vertex~$y$.
 Therefore, the rate at which vertex~$x$ switches from opinion~$a$ to opinion~$a' \neq a$ given that the process is in configuration~$\im$ is
\begin{equation}
\label{eq:imitation}
  c_{a \to a'} (x, \im) = \sum_{y \in N_x} \,\ind \{\im (y) = a' \ \hbox{and} \ d (a, a') \leq \tau \}
\end{equation}
 where~$d (\cdot, \cdot)$ denotes the opinion distance, i.e., the geodesic distance on the opinion graph, and where the subset~$N_x$ is the
 interaction neighborhood of vertex~$x$ defined as
 $$ N_x = \{y \in \V : \{x, y \} \in \E \} = \{y \in \V : \vec{xy} \in \vec{\E} \}. $$
 In the second opinion model, denoted by~$(\at_t)$ from now on, rather than imitating their neighbor, individuals update their opinion so that
 it moves one unit closer to their neighbor's opinion.
 This evolution rule is reminiscent of the Axelrod
 model~\cite{axelrod_1997, lanchier_2012a, lanchier_moisson_2016, lanchier_scarlatos_2013, lanchier_schweinsberg_2012}, and the vectorial Deffuant
 model~\cite{deffuant_al_2001, lanchier_scarlatos_2014}, where each individual is characterized by a vector of cultural features, and each interaction
 results in an individual imitating only one of her neighbor's cultural features.
 To define the dynamics of the stochastic process~$(\zeta_t)$, we introduce the sets
 $$ D (a, b) = \{a' \in V : d (a, a') = 1 \ \hbox{and} \ d (a', b) = d (a, b) - 1 \} \quad \hbox{for all} \quad a \neq b. $$
 In other words, the set~$D (a, b)$ is the set of opinions that are at opinion distance one from opinion~$a$ and one unit closer to opinion~$b$.
 We point out that, because the opinion graph is assumed to be connected, this set is nonempty.
 Then, the rate at which vertex~$x$ switches from opinion~$a$ to opinion~$a' \neq a$ given that the process is in configuration~$\at$ is
\begin{equation}
\label{eq:attraction}
  c_{a \to a'} (x, \at) =
       \sum_{y \in N_x} \,\sum_{b \neq a} \ \frac{\ind \{\at (y) = b \ \hbox{and} \ d (a, b) \leq \tau \ \hbox{and} \ a' \in D (a, b) \}}{\card D (a, b)}.
\end{equation}
 The local transition rates indicate that, each time~$\vec{xy}$ becomes active and the two neighbors are compatible, the new opinion at vertex~$x$ is
 chosen uniformly at random from the set of opinions that are adjacent to the previous opinion at vertex~$x$ and one unit closer to the opinion at
 vertex~$y$.
 From now on, we call {\bf imitation process}, respectively, {\bf attraction process}, the process described by the local transition
 rates~\eqref{eq:imitation}, respectively, \eqref{eq:attraction}.
 As previously mentioned, the imitation rule was introduced in~\cite{lanchier_scarlatos_2017} where the authors studied fluctuation and fixation of
 the process when the spatial graph is the one-dimensional integer lattice.
 Although the current paper is the first paper studying the attraction rule, the attraction process was invented a couple of years ago by~Lanchier
 and~Scarlatos during their discussions while they were writing~\cite{lanchier_scarlatos_2017}. \vspace*{-5pt} \\


\noindent {\bf Main results.}
 Our main results give lower bounds for the probability of consensus for both processes.
 These lower bounds are uniform over all possible finite connected spatial graphs.
 For the imitation process, these lower bounds also hold for all possible finite connected opinion graphs, whereas for the attraction process,
 the result only holds when the opinion graph satisfies a certain set of inequalities that we shall call eccentricity inequalities.
 To define mathematically what is meant by consensus, we first define the {\bf time to fixation} as
 $$ T = \sup \{t : \im_t \neq \im_{t-} \} = \inf \{t : \im_t = \im_s \ \hbox{for all} \ s > t \}, $$
 the last time at which the imitation process is updated.
 The time to fixation under the attraction rule is defined similarly, replacing~$\im$ by~$\at$.
 We will prove that the time to fixation is an almost surely finite stopping time for the natural filtration of the opinion model.
 Then, we define {\bf consensus} as the event that, at the time to fixation, all the individuals share the same opinion:
 $$ \{\mathrm{consensus} \} = \{\im_T \equiv \hbox{cst} \} = \{\im_T \equiv c \ \hbox{for some} \ c \in V \}. $$
 Again, the consensus event under the attraction rule is defined similarly, replacing~$\im$ by~$\at$.
 To state our results, recall that the {\bf eccentricity} of a vertex~$a \in V$, and the {\bf radius}~$\r$ and {\bf diameter}~$\d$ of the opinion
 graph are defined respectively as
 $$ \ep (a) = \max_{b \in V} \,d (a, b), \quad \r = \min_{a \in V} \,\ep (a) \quad \hbox{and} \quad \d = \max_{a \in V} \,\ep (a). $$
 When the confidence threshold~$\tau$ is at least equal to the diameter~$\d$, all the individuals can interact with their neighbors just like in
 the voter model, which always results in a consensus.
\begin{theorem} --
\label{th:consensus}
 Assume~$\tau \geq \d$.
 Then, $P (\im_T \equiv \mathrm{cst}) = P (\at_T \equiv \mathrm{cst}) = 1$.
\end{theorem}
 We now study the two spatial opinion models in the nontrivial case where~$\tau < \d$.
 Our analysis, however, only holds when~$\tau \geq \r$.
 For each process, we first give a general lower bound that applies to all possible spatial graphs and all possible initial configurations, and
 then give a more explicit expression of the lower bound when the opinions at different vertices are initially independent and equally likely, i.e.,
 for all~$\V' \subset \V$ and all~$(a_x)_{x \in \V'} \subset V$, we have
 $$ \begin{array}{rcl}
      P (\im_0 (x) = a_x \ \hbox{for all} \ x \in \V') & \n = \n & P (\at_0 (x) = a_x \ \hbox{for all} \ x \in \V') \vspace*{4pt} \\
                                                       & \n = \n & (1 / \card (V))^{\card (\V')}. \end{array} $$
 This initial distribution is referred to as the {\bf uniform product measure} later.
 To find a lower bound for the probability of consensus under the imitation rule, we will prove that the process~$(\imm_t)$ that keeps track of the
 number of individuals whose opinion's eccentricity does not exceed the confidence threshold is a martingale.
 An application of the optional stopping theorem to this process stopped at the time to fixation gives the following lower bound.
\begin{theorem}[imitation] --
\label{th:imitation}
 Assume~$\tau \in [\r, \d)$. Then,
 $$ P (\im_T \equiv \mathrm{cst}) \geq \frac{E (\card \{x \in \V : \ep (\im_0 (x)) \leq \tau \})}{\card (\V)}. $$
 In particular, starting from the uniform product measure,
 $$ P (\im_T \equiv \mathrm{cst}) \geq \frac{\card \{a \in V : \ep (a) \leq \tau \}}{\card (V)} > 0. $$
\end{theorem}
 To find a lower bound for the probability of consensus under the attraction rule, we will study a process~$(\att_t)$ that keeps track of the
 eccentricity of the individuals' opinions.
 We will prove that, whenever the opinion graph satisfies the inequalities
 \begin{equation}
\label{eq:ecc-ineq}
  \sum_{a' \in D (a, b)} \frac{\ep (a') - \ep (a)}{\card D (a, b)} + \sum_{b' \in D (b, a)} \frac{\ep (b') - \ep (b)}{\card D (b, a)} \leq 0
  \quad \hbox{for all} \quad a, b \in V,
\end{equation}
 that we shall call {\bf eccentricity inequalities}, interactions among individuals tend to decrease the eccentricity of their opinion, which
 translates mathematically into the fact that~$(\att_t)$ is a supermartingale.
 The lower bound for the probability of consensus again follows from an application of the optional stopping theorem to the process stopped at
 the fixation time.
 More precisely, for all opinion graphs that satisfy~\eqref{eq:ecc-ineq}, we have the following lower bound.
\begin{theorem}[attraction] --
\label{th:attraction}
 Assume~$\tau \in [\r, \d)$ and~\eqref{eq:ecc-ineq}. Then,
 $$ P (\at_T \equiv \mathrm{cst}) \geq 1 - \frac{1}{\card (\V)} \sum_{x \in \V} \bigg(\frac{E (\ep (\at_0 (x))) - \r}{\tau + 1 - \r} \bigg). $$
 In particular, starting from the uniform product measure,
 $$ P (\at_T \equiv \mathrm{cst}) \geq 1 - \frac{1}{\card (V)} \ \sum_{a \in V} \bigg(\frac{\ep (a) - \r}{\tau + 1 - \r} \bigg). $$
\end{theorem}
 Note that the lower bounds in Theorems~\ref{th:imitation} and~\ref{th:attraction} are universal in the sense that they do not depend on
 the topology of the spatial graph.
 Note also that, at least starting from the uniform product measure and when~$\tau \geq \r$, the lower bound in Theorem~\ref{th:imitation} is
 always positive.
 Theorem~\ref{th:attraction}, however, can lead to trivial lower bounds that are negative.
 To show that the theorem indeed implies that consensus occurs with positive probability in nontrivial cases where~$\tau < \d$, we now apply
 the theorem to specific opinion graphs.
\begin{figure}[t]
\centering
\scalebox{0.50}{\input{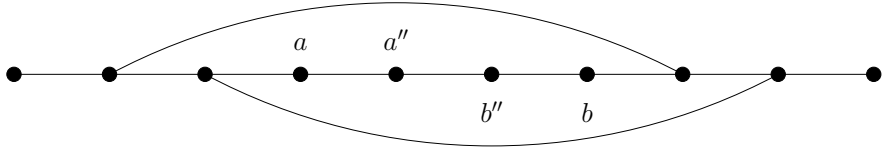}}
\caption{\upshape{Example of a graph for which~\eqref{eq:ecc-ineq} is not satisfied.
                  In this example, the sets~$D (a, b)$ and~$D (b, a)$ reduce to the singletons~$\{a'' \}$ and~$\{b'' \}$, respectively,
                  and~$\ep (a) = \ep (b) = 3 < 4 = \ep (a'') = \ep (b'')$.}}
\label{fig:counter-ex}
\end{figure}
 Before looking at these particular graphs, we point out that there exist graphs that do not satisfy~\eqref{eq:ecc-ineq}, and we refer
 to Figure~\ref{fig:counter-ex} for a counter-example. \vspace*{5pt} \\
\noindent {\bf Lattices}.
 To begin with, we assume that the set of opinions is
\begin{equation}
\label{eq:lattice}
  V = \bigg(\prod_{i = 1}^n \ [- L_i, L_i] \bigg) \cap \Z^n \quad \hbox{where} \quad L_i \in \N^*
\end{equation}
 which we turn into a graph by connecting two vertices by an edge if and only if they are at~Euclidean distance one apart.
 See the first picture in Figure~\ref{fig:graphs} for an illustration.
 In this case, each opinion can be seen as a culture, a vector of~$n$ cultural features that assume different states, just like in the
 Axelrod model~\cite{axelrod_1997}.
 After proving that the lattice satisfies~\eqref{eq:ecc-ineq}, an application of Theorem~\ref{th:attraction} gives the following lower bound
 for the probability of consensus.
\begin{figure}[t]
\centering
\scalebox{0.42}{\input{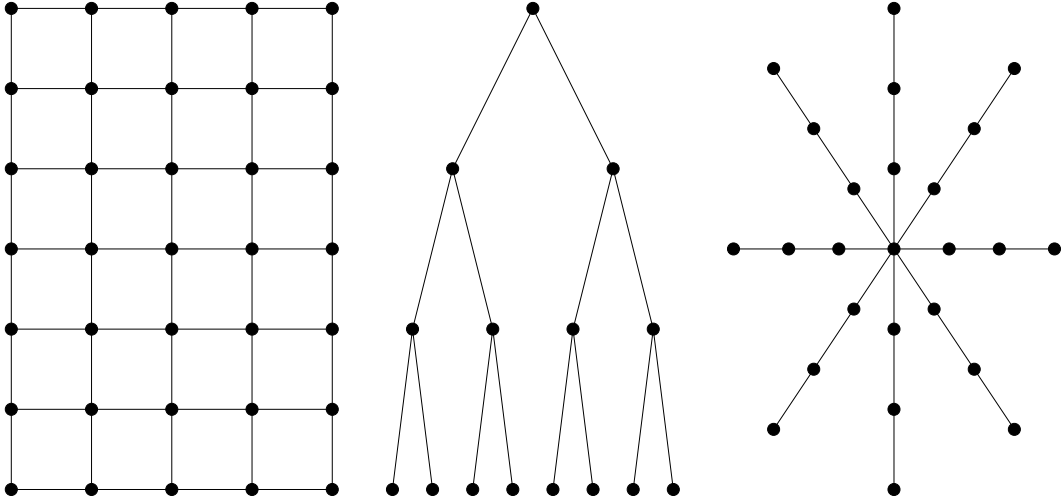}}
\caption{\upshape{Picture of the graphs in Theorems~\ref{th:lattice}--\ref{th:star}.}}
\label{fig:graphs}
\end{figure}
\begin{theorem}[lattices] --
\label{th:lattice}
 Assume that~$\tau \in [\r, 2 \r)$, that~$G$ is the lattice with vertex set~\eqref{eq:lattice} and that the process starts from the uniform
 product measure. Then,
 $$ P (\at_T \equiv \mathrm{cst}) \geq 1 - \bigg(\frac{1}{\tau + 1 - \r} \bigg) \sum_{i = 1}^n \ \frac{L_i (L_i + 1)}{2L_i + 1} $$
 where~$\r = L_1 + L_2 + \cdots + L_n$.
\end{theorem}
 For instance, when~$L_i \equiv L$, the radius is~$\r = nL$, and some basic algebra shows that the lower bound in the theorem is
 nontrivial (strictly positive) if and only if
\begin{equation}
\label{eq:cube}
  \tau \geq \bigg\lfloor \r - 1 + \frac{nL (L + 1)}{2L + 1} \bigg\rfloor + 1
\end{equation}
 where~$\lfloor x \rfloor$ denotes the integer part of~$x \in \R$.
 Similarly, when the dimension~$n = 2$, the lower bound in the theorem is nontrivial (strictly positive) if and only if
\begin{equation}
\label{eq:rect}
  \tau \geq \bigg\lfloor \r - 1 + \frac{L_1 (L_1 + 1)}{2L_1 + 1} + \frac{L_2 (L_2 + 1)}{2L_2 + 1} \bigg\rfloor + 1.
\end{equation}
 The right-hand sides of~\eqref{eq:cube} and~\eqref{eq:rect} are computed in Tables~\ref{tab:graphs}.a and~\ref{tab:graphs}.b, respectively,
 for various values of the parameters of the lattice.
 Note that the numbers in the first, respectively, second, table are significantly less than~$\d = 2nL$, respectively, $\d = 2L_1 + 2L_2$.
 This shows that there is a wide range of confidence thresholds less than the diameter for which the bound in the theorem is positive, i.e., there is
 a positive probability of consensus in a number of nontrivial cases. \vspace*{5pt} \\
\noindent {\bf Trees}.
 We now assume that the opinion graph consists of the regular rooted tree with degree~$n$ and radius~$r$, and we refer to the second
 picture in Figure~\ref{fig:graphs} for an illustration of the opinion graph that has~$n = 2$ and~$r = 3$.
 In this case, the root plays the role of a ``centrist'' opinion, while the leaves are ``extremist'' opinions.
 Like lattices, trees satisfy the eccentricity inequalities so we can apply Theorem~\ref{th:attraction}, which gives the following lower bound.
\begin{theorem}[trees] --
\label{th:tree}
 Assume that~$\tau \in [r, 2r)$, that~$G$ is the regular rooted tree with degree~$n \geq 2$ and radius~$r$, and that the process starts from the
 uniform product measure. Then,
 $$ \begin{array}{rcl}
      P (\at_T \equiv \mathrm{cst}) & \n \geq \n &
         \displaystyle 1 - \bigg(\frac{1}{\tau + 1 - r} \bigg) \bigg(\frac{n + 2n^2 + \cdots + rn^r}{1 + n + n^2 + \cdots + n^r} \bigg) \vspace*{8pt} \\ & \n = \n &
         \displaystyle 1 - \bigg(\frac{1}{\tau + 1 - r} \bigg) \ \frac{n (rn^{r + 1} - (r + 1) \,n^r + 1)}{(1 - n)(1 - n^{r + 1})}. \end{array} $$
\end{theorem}
 When~$\tau = 2r$, the diameter of the tree, there is almost sure consensus starting from any initial configuration because all the
 individuals can interact regardless of their opinion.
 When~$\tau = 2r - 1$, the lower bound in the theorem becomes
 $$ 1 - \frac{1}{r} \ \frac{n + 2n^2 + \cdots + rn^r}{1 + n + n^2 + \cdots + n^r} \geq
    1 - \frac{n + 2n^2 + \cdots + rn^r}{rn + rn^2 + \cdots + rn^r} > 0, $$
 therefore the probability of consensus is positive in the nontrivial case~$\tau = 2r - 1$.
 However, the lower bound becomes negative when~$\tau = 2r - 2$.
 The lower bound from the theorem in the nontrivial case~$\tau = 2r - 1$ is computed in Table~\ref{tab:graphs}.c for various values of~$r$ and~$n$. \vspace*{5pt} \\
\noindent {\bf Star-like graphs}.
 Finally, we assume that the opinion graph consists of the star-like graph with~$n$ branches and radius~$r$ depicted in the third picture
 of Figure~\ref{fig:graphs} when~$n = 8$ and~$r = 3$.
 In this case, each individual can choose among~$n$ ideologies and decide to be moderate (with an opinion close to the center of the graph)
 or extremist (with an opinion far from the center of the graph) in her choice.
 This graph again satisfies the eccentricity inequalities and an application of Theorem~\ref{th:attraction} now gives the following
 lower bound for the probability of consensus.
\begin{theorem}[stars] --
\label{th:star}
 Assume that~$\tau \in [r, 2r)$, that~$G$ is the star-like graph with~$n$ branches and radius~$r$, and that the process starts from the uniform
 product measure. Then,
 $$ P (\at_T \equiv \mathrm{cst}) \geq 1 - \bigg(\frac{1}{\tau + 1 - r} \bigg) \ \frac{r (r + 1) \,n}{2 (1 + rn)}. $$
\end{theorem}
 The lower bound is nontrivial (strictly positive) if and only if
\begin{equation}
\label{eq:star}
  \tau \geq \bigg\lfloor r - 1 + \frac{r (r + 1) \,n}{2 (1 + rn)} \bigg\rfloor + 1.
\end{equation}
 The right-hand side of~\eqref{eq:star} is computed in Table~\ref{tab:graphs}.d for various choices of the number~$n$ of branches and the radius~$r$.
 As for the lattice, the table shows that there is a wide range of confidence thresholds less than the diameter~$2r$ for which the bound
 in the theorem is positive, proving that there is a positive probability of consensus in a number of nontrivial cases. \\ \\
 The rest of the paper is devoted to the proofs of Theorems~\ref{th:consensus}--\ref{th:star}.
 A common aspect of the proofs of the first three theorems is to identify a certain set of configurations that will ultimately lead to consensus
 with probability one (see Lemma~\ref{lem:consensus} below).
 Because one of the key ingredients to establish~Theorems~\ref{th:imitation} and~\ref{th:attraction} is also the application of the optional stopping
 theorem, we will start by proving that the time to fixation is an almost surely finite stopping time (see Lemmas~\ref{lem:stopping}
 and~\ref{lem:finite} below).
 Another difficulty is to show that the auxiliary process~$(\imm_t)$ mentioned above, as well as the process~$(\att_t)$ when~\eqref{eq:ecc-ineq}
 holds, are bounded supermartingales adapted to the natural filtration of the two opinion models (see Lemmas~\ref{lem:martingale}
 and~\ref{lem:supermartingale} below).
 Finally, the key idea to prove Theorems~\ref{th:lattice}--\ref{th:star} is to show that the three opinion graphs in Figure~\ref{fig:graphs} satisfy
 the eccentricity inequalities~\eqref{eq:ecc-ineq} in order to use the second part of Theorem~\ref{th:attraction} (see
 Lemmas~\ref{lem:lattice-ecc}--\ref{lem:star-ecc} below).

\newpage

\begin{table}[h!]
\centering
\scalebox{0.64}{\input{tab-graphs.pstex_t}}
\caption{\upshape{Numerical estimates for the attraction rule process when the opinion graph is a lattice (top two tables), a tree (third table), and
                  a star-like graph (last table).}}
\label{tab:graphs}
\end{table}

\newpage


\section{Time to fixation}
\label{sec:fixation}
 The objective of this section is to prove that, for both spatial opinion models, the time to fixation is an almost surely finite stopping time
 for the natural filtration~$(\F_t)$ of the process under consideration.
 The main reason for showing this result is to be able later to apply the optional stopping theorem to auxiliary processes stopped at the
 time to fixation.
 To begin with, we prove the result for the attraction process, and then explain at the end of this section how to adapt the proof to the
 imitation process.
 Recall that the time to fixation is defined as
 $$ T = \sup \{t : \at_t \neq \at_{t-} \} = \inf \{t : \at_t = \at_s \ \hbox{for all} \ s > t \}, $$
 the last time at which the opinion model is updated.
 Note that it is not clear from the definition that this time is a stopping time.
 Indeed, it is not obvious that the event that~$T > t$ can be determined from the realization of the process up to time~$t$ only.
 Throughout this section, given a configuration~$\at$, edge~$\{x, y \} \in \E$ is said to be a
 $$ \begin{array}{rcl}
     \hbox{{\bf congruent} edge} & \hbox{when} & d (\at (x), \at (y)) = 0, \vspace*{4pt} \\
    \hbox{{\bf compatible} edge} & \hbox{when} & d (\at (x), \at (y)) \in \{1, 2, \ldots, \tau \}, \vspace*{4pt} \\
    \hbox{{\bf discordant} edge} & \hbox{when} & d (\at (x), \at (y)) \in \{\tau + 1, \tau + 2, \ldots, \d \}. \end{array} $$
 From now on, we let~$\phi (\at)$ be the number of compatible edges:
 $$ \phi (\at) = \sum_{\{x, y \} \in \E} \ind \{0 < d (\at (x), \at (y)) \leq \tau \}. $$
 The fact that the time to fixation is a stopping time is an immediate consequence of the fact that it is the first time the opinion model
 hits the set~$\A$ of absorbing states, which turns out to coincide with the set of configurations with only congruent and discordant edges:
 $$ \A = \{\at : \phi (\at) = 0 \}. $$
 Proving that the time to fixation is also almost surely finite is more complicated.
 The idea is to show that, from each configuration, the process can reach a configuration with less compatible edges.
 This, together with the finiteness of the spatial graph, implies that the time to fixation is almost surely finite.
 The next lemma identifies the set of absorbing states.
\begin{lemma}[absorbing states] --
\label{lem:absorbing}
 For every configuration~$\at$,
 $$ P (\at_t = \at_s \ \hbox{for all} \ s > t \,| \,\at_t = \at) = \ind \{\at \in \A \} = \ind \{\phi (\at) = 0 \}. $$
\end{lemma}
\begin{proof}
 Assume first that~$\at_t \in \A$, let
 $$ s_* = \inf \{s > t : \hbox{some edge~$\vec{wz} \in \vec{\E}$ becomes active at time~$s$} \} $$
 and let~$\vec{xy}$ be the edge becoming active at that time.
 Because~$\phi (\at) = 0$, edge~$\{x, y \}$ is either a congruent edge or a discordant edge, meaning that~$x$ and~$y$ either
 already agree, or disagree too much to influence each other.
 In either case, the opinion at~$y$ remains unchanged.
 By induction, we deduce that the configuration remains unchanged at any future time so
\begin{equation}
\label{eq:absorbing-1}
  P (\at_t = \at_s \ \hbox{for all} \ s > t \,| \,\at_t = \at) = 1 \quad \hbox{for all} \quad \at \in \A.
\end{equation}
 Assume now that~$\at_t \notin \A$.
 Then, there is at least one compatible edge so
 $$ s^* = \inf \{s > t : \hbox{some compatible edge~$\vec{wz} \in \vec{\E}$ becomes active at time~$s$} \} $$
 is stochastically smaller than the exponential distribution with mean one, and therefore almost surely finite.
 In addition, letting~$\vec{xy}$ be the edge becoming active at time~$s^*$,
\begin{equation}
\label{eq:absorbing-2}
  d (\at_{s^*} (x), \at_{s^*} (y)) = d (\at_t (x), \at_t (y)) - 1
\end{equation}
 which implies that~$\at_{s^*} \neq \at_t$. In conclusion,
\begin{equation}
\label{eq:absorbing-3}
  P (\at_t = \at_s \ \hbox{for all} \ s > t \,| \,\at_t = \at) = 0 \quad \hbox{for all} \quad \at \notin \A.
\end{equation}
 Combining~\eqref{eq:absorbing-1} and~\eqref{eq:absorbing-3} gives the desired result.
\end{proof} \\ \\
 The fact that~$T$ is a stopping time directly follows from Lemma~\ref{lem:absorbing}.
\begin{lemma} --
\label{lem:stopping}
 The time to fixation~$T$ is a stopping time.
\end{lemma}
\begin{proof}
 Let~$(\F_t)$ be the natural filtration associated to the attraction process.
 According to Lemma~\ref{lem:absorbing}, the time to fixation can be written as
\begin{equation}
\label{eq:fixation-absorbing}
  T = \inf \{t : \at_t \in \A \} = \inf \{t : \phi (\at_t) = 0 \},
\end{equation}
 from which it follows that
 $$ \{T > t \} = \{\at_s \notin \A \ \hbox{for all} \ s \leq t \} = \{\at_t \notin \A \} = \{\at_t \in \A \}^c \in \F_t. $$
 This shows that~$T$ is a stopping time.
\end{proof} \\ \\
 The next lemma shows that, from each configuration~$\at$ that has at least one compatible edge, the attraction process can reach a configuration~$\at'$
 that has less compatible edges than~$\at$, which is the first step to prove that the time to fixation is almost surely finite.
\begin{lemma} --
\label{lem:decrease}
 For every~$\at \notin \A$, there exists~$\at'$ such that
 $$ P (\at_t = \at' \,| \,\at_0 = \at) > 0 \ \hbox{for all} \ t > 0 \quad \hbox{and} \quad \phi (\at') < \phi (\at). $$
\end{lemma}
\begin{figure}[t]
\centering
\scalebox{0.50}{\input{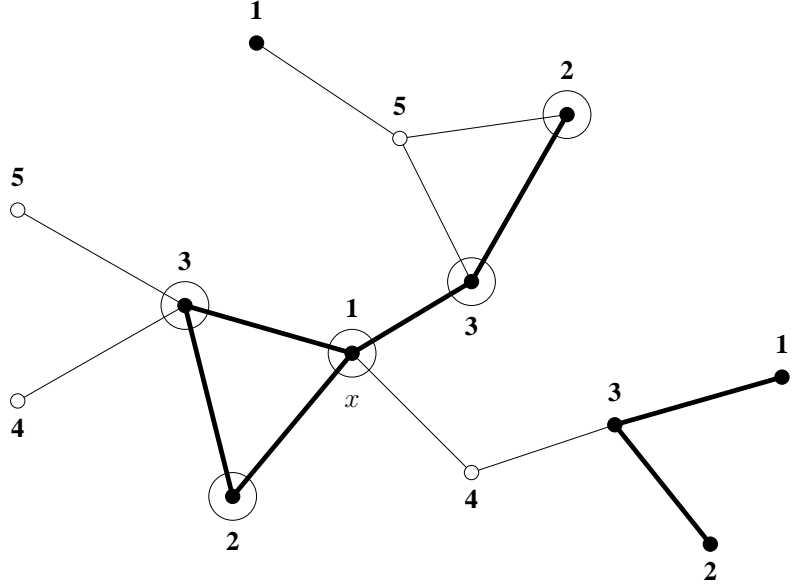}}
\caption{\upshape{Picture of the cluster~$\C_x$.
 The numbers next to the vertices represent the individuals' opinion in configuration~$\at$.
 In our example, the opinion distance we use is the Euclidean distance~$d (a, b) = |a - b|$ and the confidence threshold is~$\tau = 2$.
 Vertex~$x$ is the vertex with opinion~1 in the middle of the picture.
 In particular, the subset~$\V_x$ consists of all the vertices with opinion~1, 2, or 3, which corresponds to the black vertices.
 The subset~$\E_x$ is represented by the bold edges, while the cluster~$\C_x$ is represented by the circled vertices.
 Configuration~$\at'$ in the proof of Lemma~\ref{lem:decrease} is obtained by simply turning the opinion of all the circled vertices into opinion~1.}}
\label{fig:cluster}
\end{figure}
\begin{proof}
 Because~$\phi (\at) \neq 0$, configuration~$\at$ has at least one compatible edge, say~$\{x, y \} \in \E$.
 Then, we define the subset of vertices and the subset of edges
 $$ \V_x = \{z \in \V : d (\at (x), \at (z)) \leq \tau \} \quad \hbox{and} \quad \E_x = \{ \{w, z \} \in \E : w, z \in \V_x \}. $$
 That is, $\V_x$ is the set of individuals that are within opinion distance~$\tau$ from~$x$, and~$\E_x$ is the set of edges
 that connect individuals in~$\V_x$.
 We also let
 $$ \C_x = \hbox{connected component of the graph~$(\V_x, \E_x)$ that contains~$x$}. $$
 See Figure~\ref{fig:cluster} for a picture.
 Assuming that edge~$\vec{yx} \in \vec{\E}$ becomes active~$\tau$ times in a row before anything else happens, which occurs with positive
 probability because the graph is finite, after these interactions, the individual at vertex~$y$ has the same opinion as the initial opinion
 of the individual at vertex~$x$.
 Because the subgraph induced by~$\C_x$ is connected and finite, a simple induction implies that, for all~$t > 0$, there is a positive
 probability that, at time~$t$, all the individuals in~$\C_x$ have the same opinion as the initial opinion of the individual at vertex~$x$.
 In equations,
\begin{equation}
\label{eq:decrease-1}
  P (\at_t = \at' \,| \,\at_0 = \at) > 0 \quad \hbox{for all} \quad t > 0,
\end{equation}
 where~$\at'$ is the configuration
 $$ \at' (z) = \left\{\begin{array}{rcl}
                      \at (x) & \hbox{for all} & z \in \C_x \vspace*{3pt} \\
                      \at (z) & \hbox{for all} & z \in \V \setminus \C_x \end{array} \right. $$
 obtained from~$\at$ by replacing all the opinions in~$\C_x$ by the initial opinion~$\at (x)$.
 To compare the number of compatible edges, observe that 
 $$ \begin{array}{rcl}
      d (\at' (w), \at' (z)) = d (\at (x), \at (x)) = 0    & \hbox{whenever} & w \in \C_x \ \hbox{and} \ z \in \C_x \vspace*{4pt} \\
      d (\at' (w), \at' (z)) = d (\at (x), \at (z)) > \tau & \hbox{whenever} & w \in \C_x \ \hbox{and} \ z \in \V \setminus \C_x. \end{array} $$
 In summary, each edge~$\{w, z \} \in \E$ is a
\begin{equation}
\label{eq:decrease-2}
  \begin{array}{rcl}
   \hbox{congruent edge for~$\at'$} & \hbox{whenever} & w \in \C_x \ \hbox{and} \ z \in \C_x \vspace*{4pt} \\
  \hbox{discordant edge for~$\at'$} & \hbox{whenever} & w \in \C_x \ \hbox{and} \ z \in \V \setminus \C_x \end{array}
\end{equation}
 while it is obvious that all the other edges are of the same type for both configuration~$\at$ and configuration~$\at'$.
 Recall also from the definition of~$x$ and~$y$ that
\begin{equation}
\label{eq:decrease-3}
  \{x, y \} \ \hbox{is a compatible edge for~$\at$ but a congruent edge for~$\at'$}.
\end{equation}
 It follows from~\eqref{eq:decrease-2} and~\eqref{eq:decrease-3} that
\begin{equation}
\label{eq:decrease-4}
  \phi (\at') \leq \phi (\at) - 1 < \phi (\at).
\end{equation}
 Combining~\eqref{eq:decrease-1} and~\eqref{eq:decrease-4} implies the lemma.
\end{proof} \\ \\
 Using the previous lemma and the finiteness of the spatial graph, we can now prove that the time to fixation is almost surely finite.
\begin{lemma} --
\label{lem:finite}
 The time to fixation~$T$ is almost surely finite.
\end{lemma}
\begin{proof}
 Because the time to fixation is the first time at which the attraction process hits an absorbing state, it suffices to prove that, for
 every configuration~$\at$,
 $$ P (\at_t \in \A \,| \,\at_0 = \at) > 0 \quad \hbox{for some} \quad t > 0. $$
 In words, there exists at least one absorbing state~$\at'$ that can be reached from~$\at$.
 This is obvious if configuration~$\at \in \A$.
 Otherwise, applying Lemma~\ref{lem:decrease} a finite number of times (at most as many times as there are edges in the spatial graph) gives
 a sequence of configurations
 $$ \at = \at^0, \at^1, \at^2, \ldots, \at' = \at^n \quad \hbox{with} \quad n \leq \card (\E) $$
 such that, for all~$k = 0, 1, \ldots, n - 1$ and all~$t > 0$,
 $$ P (\at_t = \at^{k + 1} \,| \,\at_0 = \at^k) > 0 \quad \hbox{and} \quad \phi (\at^0) < \phi (\at^1) < \cdots < \phi (\at^n) = 0. $$
 Taking~$t = 1$ and using Chapman-Kolmogorov's equations, we get
 $$ \begin{array}{l}
    \displaystyle P (\at_n = \at' \,| \,\at_0 = \at) \geq \prod_{k = 0}^{n - 1} P (\at_{k + 1} = \at^{k + 1} \,| \,\at_k = \at^k) \vspace*{-4pt} \\ \hspace*{120pt} =
    \displaystyle \prod_{k = 0}^{n - 1} P (\at_1 = \at^{k + 1} \,| \,\at_0 = \at^k) > 0 \end{array} $$
 where~$\phi (\at') = \phi (\at^n) = 0$ and so~$\at' \in \A$.
 This completes the proof.
\end{proof} \\ \\
 The proofs of Lemmas~\ref{lem:absorbing}--\ref{lem:finite} easily extend to the imitation process with only two exceptions.
 First, because now two neighbors completely agree after an interaction, equation~\eqref{eq:absorbing-2} becomes
 $$ 0 = d (\im_{s^*} (x), \im_{s^*} (y)) < 1 \leq d (\im_t (x), \im_t (y)) $$
 therefore Lemma~\ref{lem:absorbing} also holds for the imitation process.
 Second, directed edge~$\vec{yx} \in \vec{\E}$ in the proof of Lemma~\ref{lem:decrease} only needs to become active
 once for vertex~$y$ to have the same opinion as the initial opinion at vertex~$x$.
 In particular, this lemma also holds for the imitation process.
 In conclusion, the time to fixation is an almost surely finite stopping time for both processes.


\section{Almost sure consensus}
\label{sec:consensus}
 This section is devoted to the proof of Theorem~\ref{th:consensus} which states almost sure consensus for both processes when~$\tau \geq \d$,
 the diameter of the opinion graph.
 Because the proof is exactly the same for the imitation process and the attraction process, we only show the result for the latter.
 We start with a general lemma that will be used to prove Theorems~\ref{th:consensus}--\ref{th:attraction}.
\begin{lemma} --
\label{lem:consensus}
 We have the inclusion of events
 $$ A = \{\ep (\at_T (x)) \leq \tau \ \hbox{for some} \ x \in \V \} \subset \{\at_T \equiv \mathrm{cst} \}. $$
\end{lemma}
\begin{proof}
 Assume that~$A$ occurs.
 Then, there exists a vertex~$x$ such that~$\ep (\at_T (x)) \leq \tau$.
 In particular, for every neighbor~$y$ of vertex~$x$, we must have
\begin{equation}
\label{eq:consensus-1}
  d (\at_T (x), \at_T (y)) \leq \max_{b \in V} \,d (\at_T (x), b) = \ep (\at_T (x)) \leq \tau.
\end{equation}
 In addition, according to Lemma~\ref{lem:absorbing}, configuration~$\at_T$ is an absorbing state so it does not have any compatible edge,
 therefore we must have
\begin{equation}
\label{eq:consensus-2}
  \at_T (x) = \at_T (y) \quad \hbox{or} \quad d (\at_T (x), \at_T (y)) > \tau.
\end{equation}
 It follows from~\eqref{eq:consensus-1} and~\eqref{eq:consensus-2} that~$\at_T (x) = \at_T (y)$ for all the neighbors~$y$ of vertex~$x$.
 Because the spatial graph is connected, a simple induction implies that
 $$ \at_T (x) = \at_T (z) \quad \hbox{for all} \quad z \in \V, $$
 meaning that~$\at_T \equiv \mathrm{cst}$.
 This proves the lemma.
\end{proof} \\ \\
 The proof of Theorem~\ref{th:consensus} is now straightforward. \\ \\
\begin{proofof}{Theorem~\ref{th:consensus}}
 Assume that~$\tau \geq \d$. Then,
 $$ \ep (\at_t (x)) \leq \max_{a \in V} \,\ep (a) = \d \leq \tau \quad \hbox{with probability one} $$
 for all~$x \in \V$ and~$t \geq 0$.
 This and Lemma~\ref{lem:consensus} imply that
 $$ \begin{array}{l}
      P (\at_T \equiv \mathrm{cst}) \geq P (\ep (\at_T (x)) \leq \tau \ \hbox{for some} \ x \in \V) \vspace*{4pt} \\ \hspace*{100pt}
                                    \geq P (\ep (\at_T (x)) \leq \tau \ \hbox{for all} \ x \in \V) = 1. \end{array} $$
 This completes the proof.
\end{proofof}


\section{Consensus under the imitation rule}
\label{sec:imitation}
 This section is devoted to the proof of Theorem~\ref{th:imitation}.
 The key ingredient is to apply the optional stopping theorem to the process that keeps track of the number of individuals whose opinion has
 eccentricity at most the confidence threshold, i.e.,
 $$ \imm_t = \sum_{x \in \V} \ \ind \{\ep (\im_t (x)) \leq \tau \} = \card \{x \in \V : \ep (\im_t (x)) \leq \tau \}, $$
 stopped at the time to fixation.
 In order to apply the optional stopping theorem, we first prove that this process is a bounded martingale.
\begin{lemma} --
\label{lem:martingale}
 The process~$(\imm_t)$ is a bounded martingale.
\end{lemma}
\begin{proof}
 The intuition behind the proof is the following.
 The process can only increase or decrease by one, which occurs each time a compatible edge connecting an individual with opinion's eccentricity
 larger than~$\tau$ and an individual with opinion's eccentricity at most~$\tau$ becomes active.
 When this happens, each of the two individuals is equally likely to mimic the other individual, so the process is equally likely to increase
 or decrease by one and thus is a martingale.
 To turn this heuristics into equations, because each directed edge becomes active at rate one,
\begin{equation}
\label{eq:martingale-1}
  \begin{array}{l} \lim_{s \downarrow 0} \ \displaystyle \frac{P (\imm_{t + s} - \imm_t = + 1 \,| \,\F_t)}{s} \vspace*{4pt} \\ \hspace*{30pt} =
  \displaystyle \sum_{\vec{yx} \in \vec{\E}} \ind \{0 < d (\im_t (x), \im_t (y)) \leq \tau \} \
  \ind \{\ep (\im_t (x)) > \tau, \,\ep (\im_t (y)) \leq \tau \}. \end{array}
\end{equation}
 The first indicator function expresses the fact that the two individuals are compatible, while the second indicator function insures that
 the eccentricity of the opinion at~$x$ switches from larger than~$\tau$ to at most~$\tau$ so that~$(\imm_t)$ indeed increases by one.
 Similarly, we have
 $$ \begin{array}{l} \lim_{s \downarrow 0} \ \displaystyle \frac{P (\imm_{t + s} - \imm_t = - 1 \,| \,\F_t)}{s} \vspace*{4pt} \\ \hspace*{30pt} =
    \displaystyle \sum_{\vec{yx} \in \vec{\E}} \ind \{0 < d (\im_t (x), \im_t (y)) \leq \tau \} \
    \ind \{\ep (\im_t (x)) \leq \tau, \,\ep (\im_t (y)) > \tau \}. \end{array} $$
 Exchanging the roles of~$x$ and~$y$, this can be rewritten as
\begin{equation}
\label{eq:martingale-2}
  \sum_{\vec{yx} \in \vec{\E}} \ind \{0 < d (\im_t (y), \im_t (x)) \leq \tau \} \
  \ind \{\ep (\im_t (y)) \leq \tau, \,\ep (\im_t (x)) > \tau \}.
\end{equation}
 By symmetry of the graph distance, \eqref{eq:martingale-2} equals the right-hand side of~\eqref{eq:martingale-1}, so
 $$ \begin{array}{l} \lim_{s \downarrow 0} \ \displaystyle \frac{E (\imm_{t + s} - \imm_t \,| \,\F_t)}{s} \vspace*{4pt} \\ \hspace*{50pt} =
    \displaystyle \lim_{s \downarrow 0} \bigg(\displaystyle \frac{P (\imm_{t + s} - \imm_t = + 1\,| \,\F_t)}{s} -
    \displaystyle \frac{P (\imm_{t + s} - \imm_t = - 1\,| \,\F_t)}{s} \bigg) = 0, \end{array} $$
 showing that~$(\imm_t)$ is a martingale.
 The fact that the process is also bounded directly follows from the fact that the spatial graph is finite: $|X_t| \leq \card (\V) < \infty$.
\end{proof}
\begin{lemma} --
\label{lem:range}
 The range of~$\imm_T$ reduces to~$\{0, \card (\V) \}$.
\end{lemma}
\begin{proof}
 Assume by contradiction that
\begin{equation}
\label{eq:range}
  \imm_T \in \{1, 2, \ldots, \card (\V) - 1 \}.
\end{equation}
 Then, there exist~$x, y \in \V$ such that, at time~$T$, the opinion of the individual at~$x$ is at most~$\tau$ while the opinion
 of the individual at~$y$ is more than~$\tau$.
 Because the spatial graph is connected, one of the paths connecting~$x$ and~$y$ must contain an edge~$\{x', y' \} \in \E$ such that
 $$ \ep (\im_T (x')) \leq \tau \quad \hbox{and} \quad \ep (\im_T (y')) > \tau. $$
 This implies that edge~$\{x', y' \}$ is compatible at time~$T$ since
 $$ 1 \leq d (\im_T (x'), \im_T (y')) \leq \max_{b \in V} \,d (\im_T (x'), b) = \ep (\im_T (x')) \leq \tau. $$
 Hence, $\phi (\im_T) \neq 0$, which contradicts~\eqref{eq:fixation-absorbing}, therefore~\eqref{eq:range} is false.
\end{proof} \\ \\
 Using our preliminary results, we can now prove Theorem~\ref{th:imitation}. \\ \\
\begin{proofof}{Theorem~\ref{th:imitation}}
 By Lemmas~\ref{lem:stopping}, \ref{lem:finite} and~\ref{lem:martingale}, the process~$(\imm_t)$ is a bounded martingale and the
 time to fixation is an almost surely finite stopping time, both with respect to the natural filtration of the imitation process.
 In particular, the optional stopping theorem gives
\begin{equation}
\label{eq:imitation-1}
  E (\imm_T) = E (\imm_0) = E (\card \{x \in \V : \ep (\im_0 (x)) \leq \tau \}).
\end{equation}
 It also follows from Lemma~\ref{lem:range} that
\begin{equation}
\label{eq:imitation-2}
  \begin{array}{l}
    E (\imm_T) = \card (\V) \,P (\imm_T = \card (\V)) \vspace*{4pt} \\ \hspace*{50pt}
               = \card (\V) \,P (\ep (\im_T (x)) \leq \tau \ \hbox{for all} \ x \in \V) \end{array}
\end{equation}
 while according to Lemma~\ref{lem:consensus}, we have
\begin{equation}
\label{eq:imitation-3}
  \begin{array}{l}
    P (\im_T \equiv \hbox{cst}) \geq P (\ep (\im_T (x)) \leq \tau \ \hbox{for some} \ x \in \V) \vspace*{4pt} \\ \hspace*{100pt}
                                \geq P (\ep (\im_T (x)) \leq \tau \ \hbox{for all} \ x \in \V). \end{array}
\end{equation}
 Combining~\eqref{eq:imitation-1}--\eqref{eq:imitation-3}, we conclude that
 $$ P (\im_T \equiv \hbox{cst}) \geq \frac{E (\imm_T)}{\card (\V)} = \frac{E (\card \{x \in \V : \ep (\im_0 (x)) \leq \tau \})}{\card (\V)}, $$
 which proves the first part of the theorem.
 To prove the second part of the theorem, assume also that the process starts from the uniform product measure. Then,
 $$ \begin{array}{rcl}
      P (\im_T \equiv \hbox{cst}) & \n \geq \n &
         \displaystyle \frac{E (\card \{x \in \V : \ep (\im_0 (x)) \leq \tau \})}{\card (\V)} \vspace*{8pt} \\ & \n = \n &
         \displaystyle \frac{\card (\V) \,P (\ep (\im_0 (x)) \leq \tau)}{\card (\V)} = P (\ep (U) \leq \tau) \end{array} $$
 where~$U = \uniform (V)$. In particular,
 $$ P (\im_T \equiv \hbox{cst}) \geq \sum_{a \in V} \ \ind \{\ep (a) \leq \tau \} \,P (U = a) = \frac{\card \{a \in V : \ep (a) \leq \tau \}}{\card (V)}. $$
 This completes the proof.
\end{proofof}


\section{Consensus under the attraction rule}
\label{sec:attraction}
 This section is devoted to the proof of Theorem~\ref{th:attraction} and follows the same structure as the previous section.
 As previously mentioned, the key is to define an auxiliary process~$(\att_t)$ that keeps track of the eccentricity of the individuals' opinions.
 Provided the eccentricity inequalities are satisfied, this process is a supermartingale, and the theorem will follow from applying
 the optional stopping theorem to this process stopped at the time to fixation.
 The process is defined from the eccentricity of the individuals' opinions and the radius~$\r$ of the opinion graph by setting
 $$ \att_t = \sum_{x \in \V} \ (\ep (\at_t (x)) - \r) = \sum_{a \in V} \ (\ep (a) - \r) \,\card \{x \in \V : \at_t (x) = a \}. $$
 The process is smaller when the individuals' opinions are closer to the center of the opinion graph, the set of vertices whose
 eccentricity is equal to the radius, whereas it is larger when the opinions are closer to the ``periphery''.
 The fact that~$(\att_t)$ is a supermartingale essentially means that interactions among individuals tend to decrease the eccentricity of their opinions.
 To apply the optional stopping theorem later, we first prove that the process is positive and bounded.
\begin{lemma} --
\label{lem:bounded}
 For all~$t \geq 0$, we have~$0 \leq \att_t \leq \r \card (\V) < \infty$.
\end{lemma}
\begin{proof}
 Using that~$\ep (a) \geq \r$ for all~$a \in V$, we get
 $$ \att_t = \sum_{a \in V} \ (\ep (a) - \r) \,\card \{x \in \V : \at_t (x) = a \} \geq 0. $$
 Now, let~$c \in V$ be any vertex in the center of the opinion graph, meaning that~$\ep (c) = \r$.
 Then, it follows from the triangle inequality that, for all opinions~$a, b \in V$,
 $$ d (a, b) \leq d (a, c) + d (c, b) \leq \ep (c) + \ep (c) = 2 \r. $$
 This implies that~$\ep (a) = \max_{b \in V} d (a, b) \leq 2 \r$ for all~$a \in V$ so
 $$ \att_t = \sum_{x \in \V} \ (\ep (\at_t (x)) - \r) \leq \sum_{x \in \V} \ (2 \r - \r) = \r \card (\V), $$
 which is finite because both the opinion and the spatial graphs are finite.
\end{proof} \\ \\
 To prove that the process~$(\att_t)$ is a supermartingale when inequalities~\eqref{eq:ecc-ineq} hold, it is useful to write explicitly its
 infinitesimal variation.
 This is done in the next lemma.
\begin{lemma} --
\label{lem:variation}
 For all times~$t \geq 0$,
 $$ \begin{array}{l} \lim_{s \downarrow 0} \ \displaystyle \frac{E (\att_{t + s} - \att_t \,| \,\F_t)}{s} \vspace*{4pt} \\ \hspace*{20pt} =
    \displaystyle \sum_{\vec{yx} \in \vec{\E}} \bigg(\ind \{0 < d (\at_t (x), \at_t (y)) \leq \tau \}
    \sum_{c \in D (\at_t (x), \at_t (y))} \frac{\ep (c) - \ep (\at_t (x))}{\card D (\at_t (x), \at_t (y))} \bigg). \end{array} $$
\end{lemma}
\begin{proof}
 Because each directed edge~$\vec{yx} \in \vec{\E}$ becomes active at rate one, which results in the individual at vertex~$x$ updating her opinion
 based on the opinion of the individual at vertex~$y$ provided both individuals are compatible, i.e., $0 < d (\at_t (x), \at_t (y)) \leq \tau$, we have
\begin{equation}
\label{eq:variation-1}
  \begin{array}{l}
  \lim_{s \downarrow 0} \ \displaystyle \frac{E (\att_{t + s} - \att_t \,| \,\F_t)}{s} \vspace*{4pt} \\ \hspace*{20pt} =
  \displaystyle \sum_{\vec{yx} \in \vec{\E}} \ind \{0 < d (\at_t (x), \at_t (y)) \leq \tau \} \ E (\phi_{xy} (\att_t) - \att_t \,| \,\F_t) \end{array}
\end{equation}
 where~$\phi_{xy} (\att_t)$ is the new value of the process~$(\att_t)$ after individual~$x$ updates her opinion based on the opinion of individual~$y$.
 In addition, because the new opinion at vertex~$x$ is chosen uniformly at random from the subset of opinions~$D (\at_t (x), \at_t (y))$, we get
 $$ \begin{array}{l}
    \displaystyle \phi_{xy} (\att_t) - \att_t = \sum_{z \neq x} \ (\ep (\at_t (z)) - \r) + (\ep (U) - \r) \vspace*{-4pt} \\ \hspace*{120pt} -
    \displaystyle \sum_{z \in \V} \ (\ep (\at_t (z)) - \r) = \ep (U) - \ep (\at_t (x)) \end{array} $$
 where~$U = \uniform (D (\at_t (x), \at_t (y)))$.
 It follows that
\begin{equation}
\label{eq:variation-2}
  E (\phi_{xy} (\att_t) - \att_t \,| \,\F_t) = \sum_{c \in D (\at_t (x), \at_t (y))} \frac{\ep (c) - \ep (\at_t (x))}{\card D (\at_t (x), \at_t (y))}.
\end{equation}
 Combining~\eqref{eq:variation-1} and~\eqref{eq:variation-2} gives the result.
\end{proof}
\begin{lemma} --
\label{lem:supermartingale}
 Assume~\eqref{eq:ecc-ineq}. Then, $(\att_t)$ is a supermartingale.
\end{lemma}
\begin{proof}
 Pairing the directed edges
 $$ \vec{xy} \in \vec{\E} \ \ \hbox{and} \ \ \vec{yx} \in \vec{\E} \quad \hbox{for all} \quad \{x, y \} \in \E, $$
 and applying~Lemma~\ref{lem:variation}, we obtain
 $$ \begin{array}{l} \lim_{s \downarrow 0} \ \displaystyle \frac{E (\att_{t + s} - \att_t \,| \,\F_t)}{s} \vspace*{4pt} \\ \hspace*{10pt} = \
    \displaystyle \sum_{\{x, y \} \in \E} \bigg(\ind \{0 < d (\at_t (x), \at_t (y)) \leq \tau \}
                  \sum_{a' \in D (\at_t (x), \at_t (y))} \frac{\ep (a') - \ep (\at_t (x))}{\card D (\at_t (x), \at_t (y))} \vspace*{4pt} \\ \hspace*{80pt} + \
    \displaystyle \ind \{0 < d (\at_t (y), \at_t (x)) \leq \tau \}
                  \sum_{b' \in D (\at_t (y), \at_t (x))} \frac{\ep (b') - \ep (\at_t (y))}{\card D (\at_t (y), \at_t (x))} \bigg). \end{array} $$
 In view of~\eqref{eq:ecc-ineq} and because the opinion graph distance is symmetric, the right-hand side is clearly nonpositive, which implies that
 the process~$(\att_t)$ is a supermartingale.
\end{proof} \\ \\
 The last step before proving the theorem is to show how the probability of consensus under the attraction rule relates to the expected value of~$\att_T$.
\begin{lemma} --
\label{lem:lower-bound}
 For all~$\tau \in [\r, \d)$,
 $$ P (\at_T \equiv \mathrm{cst}) \geq 1 - \frac{1}{\card (\V)} \ \frac{E (\att_T)}{\tau + 1 - \r}. $$
\end{lemma}
\begin{proof}
 The key to the proof is to compute the expected value of~$\att_T$ by conditioning on the partition~$\{A, B \}$ where~$A$ and~$B$ are the two events
 $$ \begin{array}{rcl}
      A & \n = \n & \{\ep (\at_T (x)) \leq \tau \ \hbox{for some} \ x \in \V \} \vspace*{4pt} \\
      B & \n = \n & \{\ep (\at_T (x)) \geq \tau + 1 \ \hbox{for all} \ x \in \V \}. \end{array} $$
 Observing that, when~$B$ occurs,
 $$ \att_T = \sum_{x \in \V} \ (\ep (\at_T (x)) - \r) \geq \sum_{x \in \V} \ (\tau + 1 - \r) = (\tau + 1 - \r) \card (\V) $$
 and using that~$\att_T \geq 0$ according to Lemma~\ref{lem:bounded}, we get
 $$ \begin{array}{l}
      E (\att_T) = E (\att_T \,| \,A) P (A) + E (\att_T \,| \,B) P (B) \vspace*{4pt} \\ \hspace*{70pt}
              \geq E (\att_T \,| \,B) P (B) \geq (\tau + 1 - \r) \card (\V) \,P (B). \end{array} $$
 This, and the fact that~$\tau \in [\r, \d)$, give the lower bound
 $$ P (A) = 1 - P (B) \geq 1 - \frac{1}{\card (\V)} \ \frac{E (\att_T)}{\tau + 1 - \r}. $$
 Recalling also that, according to Lemma~\ref{lem:consensus},
 $$ A = \{\ep (\at_T (x)) \leq \tau \ \hbox{for some} \ x \in \V \} \subset \{\at_T \equiv \mathrm{cst} \} $$
 implies the result.
\end{proof} \\ \\
 Applying Lemmas~\ref{lem:bounded}--\ref{lem:lower-bound}, we can now prove Theorem~\ref{th:attraction}. \\ \\
\begin{proofof}{Theorem~\ref{th:attraction}}
 By Lemmas~\ref{lem:bounded} and~\ref{lem:supermartingale}, the process~$(\att_t)$ is a bounded supermartingale, while by
 Lemmas~\ref{lem:stopping} and~\ref{lem:finite}, the time to fixation~$T$ is an almost surely finite stopping time with respect to the
 natural filtration of the opinion model.
 In particular, we may apply the optional stopping theorem to the supermartingale~$(\att_t)$ stopped at time~$T$ to get
 $$ E (\att_T) \leq E (\att_0). $$
 This, together with Lemma~\ref{lem:lower-bound}, implies that
 $$ P (\at_T \equiv \mathrm{cst}) \geq 1 - \frac{1}{\card (\V)} \ \frac{E (\att_T)}{\tau + 1 - \r}
                                   \geq 1 - \frac{1}{\card (\V)} \ \frac{E (\att_0)}{\tau + 1 - \r}. $$
 Finally, recalling the definition of~$\att_t$, we conclude that
 $$ \begin{array}{rcl}
      P (\at_T \equiv \mathrm{cst}) & \n \geq \n &
         \displaystyle 1 - \frac{1}{\card (\V)} \ E \bigg(\sum_{x \in \V} \ \frac{\ep (\at_0 (x)) - \r}{\tau + 1 - \r} \bigg) \vspace*{8pt} \\ & \n = \n &
         \displaystyle 1 - \frac{1}{\card (\V)} \sum_{x \in \V} \bigg(\frac{E (\ep (\at_0 (x))) - \r}{\tau + 1 - \r} \bigg), \end{array} $$
 which proves the first part of the theorem.
 To prove the second part of the theorem, assume also that the process starts from the uniform product measure. Then,
 $$ P (\at_T \equiv \hbox{cst}) \geq
         1 - \frac{E (\ep (\at_0 (x))) - \r}{\tau + 1 - \r} = 1 - \frac{E (\ep (U)) - \r}{\tau + 1 - \r} $$
 where~$U = \uniform (V)$. In particular,
 $$ \begin{array}{rcl}
      P (\at_T \equiv \hbox{cst}) & \n \geq \n &
         \displaystyle 1 - \sum_{a \in V} \bigg(\frac{\ep (a) - \r}{\tau + 1 - \r} \bigg) \,P (U = a) \vspace*{8pt} \\ & \n = \n &
         \displaystyle 1 - \frac{1}{\card (V)} \ \sum_{a \in V} \bigg(\frac{\ep (a) - \r}{\tau + 1 - \r} \bigg). \end{array} $$
 This completes the proof.
\end{proofof}


\section{Application to special opinion graphs}
\label{sec:graphs}
 This section is devoted to the proof of Theorems~\ref{th:lattice}--\ref{th:star} that give lower bounds for the probability of consensus
 under the attraction rule for the three classes of graphs in~Figure~\ref{fig:graphs}.
 The strategy is the same for all three classes of graphs:
 First, we prove that the opinion graph under consideration satisfies the eccentricity inequalities, and then we compute explicitly the
 general lower bound in the second part of Theorem~\ref{th:attraction} for this specific graph.
 To begin with, we assume that the set of opinions is given by the~$n$-dimensional integer lattice
 $$ V = \bigg(\prod_{i = 1}^n \ [- L_i, L_i] \bigg) \cap \Z^n \quad \hbox{where} \quad L_i \in \N^*. $$
\begin{lemma} --
\label{lem:lattice-ecc}
 The lattice satisfies the eccentricity inequalities~\eqref{eq:ecc-ineq}.
\end{lemma}
\begin{proof}
 For all~$a = (a_1, \ldots, a_n), b = (b_1, \ldots, b_n) \in V$, $a \neq b$, define
 $$ \Lambda = \Lambda (a, b) = \Lambda (b, a) = \{i = 1, 2, \ldots, n : a_i \neq b_i \} \neq \varnothing. $$
 The sets~$D (a, b)$ and~$D (b, a)$ are given respectively by
 $$ \begin{array}{rclcl}
      D (a, b) & \n = \n & \{a^i : i \in \Lambda \} & \hbox{where} & a^i = a + \sgn (b_i - a_i) \,e_i \vspace*{4pt} \\
      D (b, a) & \n = \n & \{b^i : i \in \Lambda \} & \hbox{where} & b^i = b + \sgn (a_i - b_i) \,e_i \end{array} $$
 where~$\sgn$ is the sign function and where~$e_i$ is the~$i$th unit vector in~$\Z^n$.
 We refer the reader to~Figure~\ref{fig:lattice} for a picture.
\begin{figure}[t]
\centering
\scalebox{0.50}{\input{lattice.pstex_t}}
\caption{\upshape{Pictures related to the proof of Lemma~\ref{lem:lattice-ecc}.}}
\label{fig:lattice}
\end{figure}
 Note that, on the lattice, the eccentricity of opinion~$a$ is
 $$ \ep (a) = d (0, a) + \r \quad \hbox{where} \quad \r = \ep (0) = L_1 + L_2 + \cdots + L_n. $$
 In particular, for all~$i \in \Lambda$, we have
 $$ \ep (a^i) - \ep (a) = d (0, a^i) - d (0, a) = |a_i + \sgn (b_i - a_i)| - |a_i|. $$
 Considering all the possible orderings of~$0, a_i$ and~$b_i$, we deduce that
\begin{equation}
\label{eq:lattice-ecc-1}
  \ep (a^i) - \ep (a) = \left\{\begin{array}{rl} + 1 & \hbox{if} \ b_i < a_i \leq 0 \ \hbox{or} \ 0 \leq a_i < b_i \vspace*{2pt} \\
                                                 - 1 & \hbox{otherwise}. \end{array} \right.
\end{equation}
 By symmetry, we have
\begin{equation}
\label{eq:lattice-ecc-2}
  \ep (b^i) - \ep (b) = \left\{\begin{array}{rl} + 1 & \hbox{if} \ a_i < b_i \leq 0 \ \hbox{or} \ 0 \leq b_i < a_i \vspace*{2pt} \\
                                                 - 1 & \hbox{otherwise}. \end{array} \right.
\end{equation}
 Because the two conditions
 $$ (b_i < a_i \leq 0 \ \ \hbox{or} \ \ 0 \leq a_i < b_i) \quad \hbox{and} \quad (a_i < b_i \leq 0 \ \ \hbox{or} \ \ 0 \leq b_i < a_i) $$
 are incompatible, the left-hand sides of~\eqref{eq:lattice-ecc-1} and~\eqref{eq:lattice-ecc-2} cannot be simultaneously positive, from
 which it follows that they add up to either 0 or $-2$, therefore
 $$ (\ep (a^i) - \ep (a)) + (\ep (b^i) - \ep (b)) \leq 0 \quad \hbox{for all} \quad i \in \Lambda. $$
 Using also that~$\card D (a, b) = \card D (b, a) = \card (\Lambda)$, we conclude that
 $$ \begin{array}{l}
    \displaystyle \sum_{a' \in D (a, b)} \frac{\ep (a') - \ep (a)}{\card D (a, b)} +
                  \sum_{b' \in D (b, a)} \frac{\ep (b') - \ep (b)}{\card D (b, a)} \vspace*{4pt} \\ \hspace*{20pt} = \
    \displaystyle \sum_{i \in \Lambda} \ \frac{\ep (a^i) - \ep (a)}{\card (\Lambda)} +
                  \sum_{i \in \Lambda} \ \frac{\ep (b^i) - \ep (b)}{\card (\Lambda)} =
    \displaystyle \sum_{i \in \Lambda} \ \frac{(\ep (a^i) - \ep (a)) + (\ep (b^i) - \ep (b))}{\card (\Lambda)} \leq 0. \end{array} $$
 This completes the proof.
\end{proof} \\ \\
 Using Theorem~\ref{th:attraction} and Lemma~\ref{lem:lattice-ecc}, we can now prove Theorem~\ref{th:lattice}. \\ \\
\begin{proofof}{Theorem~\ref{th:lattice}}
 According to Lemma~\ref{lem:lattice-ecc}, the lattice satisfies the eccentricity inequalities~\eqref{eq:ecc-ineq}, therefore
 we may apply Theorem~\ref{th:attraction} to get
\begin{equation}
\label{eq:lattice-1}
  P (\at_T \equiv \mathrm{cst}) \geq 1 - \frac{1}{\card (V)} \ \sum_{a \in V} \bigg(\frac{\ep (a) - \r}{\tau + 1 - \r} \bigg).
\end{equation}
 Now, let~$U = \uniform (V)$ and observe that
 $$ U_i = \uniform ([- L_i, L_i] \cap \Z) \quad \hbox{for all} \quad i = 1, 2, \ldots, n. $$
 This and~$\ep (a) - \r = d (0, a) = |a_1| + \cdots + |a_n|$ imply that
\begin{equation}
\label{eq:lattice-2}
  \begin{array}{rcl}
  \displaystyle \frac{1}{\card (V)} \ \sum_{a \in V} \ (\ep (a) - \r) & \n = \n &
  \displaystyle E (\ep (U)) - \r = E (d (0, U)) = \sum_{i = 1}^n \,E |U_i| \vspace*{4pt} \\ & \n = \n &
  \displaystyle \sum_{i = 1}^n \ \frac{2 (1 + 2 + \cdots + L_i)}{2 L_i + 1} = \sum_{i = 1}^n \ \frac{L_i (L_i + 1)}{2L_i + 1}. \end{array}
\end{equation}
 Combining~\eqref{eq:lattice-1} and~\eqref{eq:lattice-2} gives
 $$ P (\at_T \equiv \mathrm{cst}) \geq 1 - \bigg(\frac{1}{\tau + 1 - \r} \bigg) \sum_{i = 1}^n \ \frac{L_i (L_i + 1)}{2L_i + 1} $$
 as desired.
\end{proofof} \\ \\
 We now assume that the opinion graph is the rooted tree with degree~$n$ and radius~$r$.
 As for the lattice, we first show that the tree satisfies the eccentricity inequalities.
\begin{lemma} --
\label{lem:tree-ecc}
 The rooted tree with degree~$n$ and radius~$r$ satisfies~\eqref{eq:ecc-ineq}.
\end{lemma}
\begin{proof}
 Letting~$a \neq b$ be two opinions, there is a unique path connecting~$a$ and~$b$, from which it follows that the sets~$D (a, b)$ and~$D (b, a)$
 are singletons, and we write
\begin{equation}
\label{eq:tree-ecc-1}
  D (a, b) = \{a'' \} \quad \hbox{and} \quad D (b, a) = \{b'' \}.
\end{equation}
 Then, we distinguish three cases.
 Assume first that~$a$ is a descendant of~$b$, meaning that the unique path going from the root~0 to vertex~$a$ crosses~$b$.
 In this case, we have
\begin{equation}
\label{eq:tree-ecc-2}
  \begin{array}{rcl}
  \ep (a'') = d (0, a'') + r & \n = \n & (d (0, a) - 1) + r = \ep (a) - 1 \vspace*{4pt} \\
  \ep (b'') = d (0, b'') + r & \n = \n & (d (0, b) + 1) + r = \ep (b) + 1. \end{array}
\end{equation}
 We refer the reader to the left-hand side of Figure~\ref{fig:tree} for an illustration of this case.
\begin{figure}[t]
\centering
\scalebox{0.50}{\input{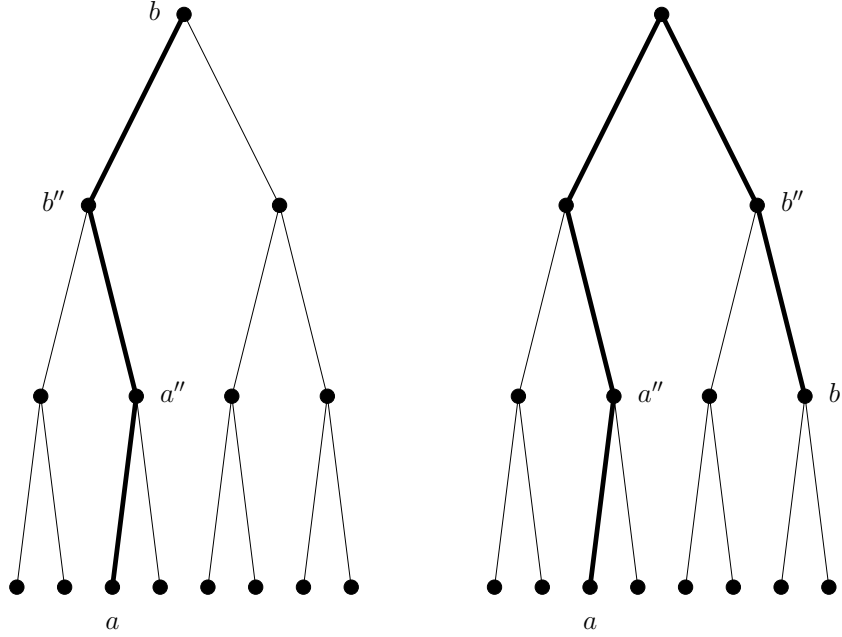}}
\caption{\upshape{Pictures related to the proof of Lemma~\ref{lem:tree-ecc}.}}
\label{fig:tree}
\end{figure}
 By symmetry, when opinion~$b$ is a descendant of opinion~$a$,
\begin{equation}
\label{eq:tree-ecc-3}
  \ep (a'') = \ep (a) + 1 \quad \hbox{and} \quad \ep (b'') = \ep (b) - 1.
\end{equation}
 Finally, when neither vertex is a descendant of the other vertex,
\begin{equation}
\label{eq:tree-ecc-4}
  \begin{array}{rcl}
  \ep (a'') = d (0, a'') + r & \n = \n & (d (0, a) - 1) + r = \ep (a) - 1 \vspace*{4pt} \\
  \ep (b'') = d (0, b'') + r & \n = \n & (d (0, b) - 1) + r = \ep (b) - 1. \end{array}
\end{equation}
 We refer the reader to the right-hand side of Figure~\ref{fig:tree} for an illustration of this case.
 Finally, combining~\eqref{eq:tree-ecc-1}--\eqref{eq:tree-ecc-4}, we conclude that, in any case,
 $$ \sum_{a' \in D (a, b)} \frac{\ep (a') - \ep (a)}{\card D (a, b)} + \sum_{b' \in D (b, a)} \frac{\ep (b') - \ep (b)}{\card D (b, a)} =
    (\ep (a'') - \ep (a)) + (\ep (b'') - \ep (b)) \leq 0 $$
 showing that the eccentricity inequalities~\eqref{eq:ecc-ineq} hold.
\end{proof} \\ \\
 We point out that, even though the proof above only applies to finite regular rooted trees, it can be extended to show that any finite tree
 satisfies the eccentricity inequalities.
 To see this, we observe that the center of the tree, defined as
 $$ C = \{c \in V : \ep (c) = \r \}, $$
 reduces to one or two vertices.
 In either case, one can extend the argument in the proof of Lemma~\ref{lem:tree-ecc} by designating the center as the root, and using that
 $$ \ep (a) = d (a, C) + \r = \min_{c \in C} \,d (a, c) + \r \quad \hbox{for all} \quad a \in V. $$
 Using Theorem~\ref{th:attraction} and Lemma~\ref{lem:tree-ecc}, we can now prove Theorem~\ref{th:tree}. \\ \\
\begin{proofof}{Theorem~\ref{th:tree}}
 According to Lemma~\ref{lem:tree-ecc}, the tree satisfies the eccentricity inequalities~\eqref{eq:ecc-ineq}, therefore
 we may apply Theorem~\ref{th:attraction} to get
\begin{equation}
\label{eq:tree-1}
  P (\at_T \equiv \mathrm{cst}) \geq 1 - \frac{1}{\card (V)} \ \sum_{a \in V} \bigg(\frac{\ep (a) - \r}{\tau + 1 - \r} \bigg)
\end{equation}
 just like in the proof of Theorem~\ref{th:lattice}. Using that
 $$ \begin{array}{rcl}
                    \ep (a) - \r = d (0, a) & \hbox{for all} & a \in V \vspace*{4pt} \\
    \card \{a \in V : d (0, a) = k \} = n^k & \hbox{for all} & k = 0, 1, \ldots, \r, \end{array} $$
 we obtain that
\begin{equation}
\label{eq:tree-2}
 \begin{array}{rcl}
 \displaystyle \frac{1}{\card (V)} \ \sum_{a \in V} \ (\ep (a) - \r) & \n = \n &
 \displaystyle \frac{1}{\card (V)} \ \sum_{a \in V} \ d (0, a) =
 \displaystyle \frac{\displaystyle \sum_{k = 0}^r \,k n^k}{\displaystyle \sum_{k = 0}^r \,n^k} \vspace*{-4pt} \\ & \n = \n &
 \displaystyle  n \ \frac{\displaystyle \frac{\partial}{\partial n} \bigg(\frac{1 - n^{r + 1}}{1 - n} \bigg)}{\displaystyle \bigg(\frac{1 - n^{r + 1}}{1 - n} \bigg)} =
 \displaystyle \frac{n (r n^{r + 1} - (r + 1) \,n^r + 1)}{(1 - n)(1 - n^{r + 1})}. \end{array}
\end{equation}
 Combining~\eqref{eq:tree-1} and~\eqref{eq:tree-2}, we conclude that
 $$ P (\at_T \equiv \mathrm{cst}) \geq 1 - \bigg(\frac{1}{\tau + 1 - r} \bigg) \ \frac{n (r n^{r + 1} - (r + 1) \,n^r + 1)}{(1 - n)(1 - n^{r + 1})}. $$
 This completes the proof.
\end{proofof} \\ \\
 Finally, we assume that the opinion graph is the star-like graph with~$n$ branches and radius~$r$ depicted on the right-hand side of Figure~\ref{fig:graphs}.
\begin{lemma} --
\label{lem:star-ecc}
 The star-like graph satisfies the eccentricity inequalities~\eqref{eq:ecc-ineq}.
\end{lemma}
\begin{proof}
 This is similar to the proof of Lemma~\ref{lem:tree-ecc}.
 Because there is a unique path connecting any two opinions, say~$a$ and~$b$, we again have
 $$ D (a, b) = \{a'' \} \quad \hbox{and} \quad D (b, a) = \{b'' \} $$
 and we distinguish three cases:
\begin{enumerate}
 \item When the path going from the center to~$a$ goes trough~$b$,
       $$ \ep (a'') = \ep (a) - 1 \quad \hbox{and} \quad \ep (b'') = \ep (b) + 1. $$
 \item When the path going from the center to~$b$ goes trough~$a$,
       $$ \ep (a'') = \ep (a) + 1 \quad \hbox{and} \quad \ep (b'') = \ep (b) - 1. $$
 \item In all other cases, neither~$a$ nor~$b$ can be the center, and the unique path connecting~$a$ and~$b$ must go through the center.
       In particular, we get
       $$ \ep (a'') = \ep (a) - 1 \quad \hbox{and} \quad \ep (b'') = \ep (b) - 1. $$
\end{enumerate}
 In all three cases, we have~$\ep (a'') + \ep (b'') \leq \ep (a) + \ep (b)$ therefore
 $$ \sum_{a' \in D (a, b)} \frac{\ep (a') - \ep (a)}{\card D (a, b)} + \sum_{b' \in D (b, a)} \frac{\ep (b') - \ep (b)}{\card D (b, a)} =
    (\ep (a'') - \ep (a)) + (\ep (b'') - \ep (b)) \leq 0 $$
 for all~$a, b \in V$, $a \neq b$.
 This completes the proof.
\end{proof} \\ \\
\begin{proofof}{Theorem~\ref{th:star}}
 By Lemma~\ref{lem:star-ecc} and Theorem~\ref{th:attraction},
\begin{equation}
\label{eq:star-1}
  P (\at_T \equiv \mathrm{cst}) \geq 1 - \frac{1}{\card (V)} \ \sum_{a \in V} \bigg(\frac{\ep (a) - \r}{\tau + 1 - \r} \bigg).
\end{equation}
 In addition, for the star-like graph, we have
 $$ \begin{array}{rcl}
    \ep (a) - \r = d (0, a) & \hbox{for all} & a \in V \vspace*{4pt} \\
    \card \{a \in V : d (0, a) = k \} = n & \hbox{for all} & k = 1, 2, \ldots, \r. \end{array} $$
 It follows that
\begin{equation}
\label{eq:star-2}
 \begin{array}{rcl}
 \displaystyle \frac{1}{\card (V)} \ \sum_{a \in V} \ (\ep (a) - \r) & \n = \n &
 \displaystyle \frac{1}{\card (V)} \ \sum_{a \in V} \ d (0, a)  \vspace*{4pt} \\ & \n = \n &
 \displaystyle \frac{(1 + 2 + \cdots + r) \,n}{1 + rn} = \frac{r (r + 1) \,n}{2 (1 + rn)}. \end{array}
\end{equation}
 Combining~\eqref{eq:star-1} and~\eqref{eq:star-2}, we conclude that
 $$ P (\at_T \equiv \mathrm{cst}) \geq 1 - \bigg(\frac{1}{\tau + 1 - r} \bigg) \ \frac{r (r + 1) \,n}{2 (1 + rn)}. $$
 This proves the theorem.
\end{proofof}


\end{document}